\documentclass{amsart}
\usepackage{amsfonts,amsthm,amsmath}
\usepackage{amssymb}
\usepackage{amscd}
\usepackage[utf8]{inputenc}
\usepackage[a4paper]{geometry}
\usepackage{enumerate}
\usepackage[usenames,dvipsnames]{color}
\usepackage{tikz}
\usepackage{soul}
\usepackage{array}
\usepackage{array,tabularx}

\numberwithin{equation}{section}
\numberwithin{equation}{subsection}

\theoremstyle{plain}

\newtheorem{theorem}[equation]{Theorem}
\newtheorem{lemma}[equation]{Lemma}
\newtheorem{proposition}[equation]{Proposition}

\newtheorem{corollary}[equation]{Corollary}

\theoremstyle{definition}
\newtheorem{example}[equation]{Example}
\newtheorem{remark}[equation]{Remark}

\newtheorem{definition}[equation]{Definition}

\newtheorem{bekezdes}[equation]{}

\newcommand{\ssw}{{\mathfrak{ sw}}}

\def\C{\mathbb C}
\def\Q{\mathbb Q}

\def\Z{\mathbb Z}

\newcommand{\calv}{{\mathcal V}}

\newcommand{\cali}{{\mathcal I}}

\newcommand{\calO}{{\mathcal O}}

\newcommand{\calS}{{\mathcal S}}

\newcommand{\calL}{\mathcal{L}}

\newcommand{\tX}{\widetilde{X}}

\newcommand{\cS}{{\mathcal S}}\newcommand{\cV}{{\mathcal V}}
\newcommand{\Pam}{{\mathrm {Path}}^\uparrow}

\newcommand{\cP}{{\mathcal P}}\newcommand{\cPa}{{\mathcal P}^\uparrow}

\newcommand{\cO}{{\mathcal O}}

\newcommand{\zbz}{Z_{B_0}}
\newcommand{\zbe}{Z_{B_1}}

\newcommand{\zm}{Z_{B_m}}

\newcommand{\zi}{Z_{B_i}}
\newcommand{\zj}{Z_{B_j}}

\newcommand{\pic}{{\rm Pic}}

\newcommand{\bt}{{\mathbf t}}

\newcommand{\bZ}{{\mathbb{Z}}}
\newcommand{\bQ}{{\mathbb{Q}}}

\author{J\'anos Nagy}
\address{Central European University, Dept. of Mathematics,  Budapest, Hungary}
\email{nagy\textunderscore janos@phd.ceu.edu}

\author{Andr\'as N\'emethi}
\address{Alfr\'ed R\'enyi Institute of Mathematics,
Hungarian Academy of Sciences,
Re\'altanoda utca 13-15, H-1053, Budapest, Hungary \newline
 \hspace*{4mm} ELTE - University of Budapest, Dept. of Geometry, Budapest, Hungary \newline \hspace*{4mm}
BCAM - Basque Center for Applied Math.,
Mazarredo, 14 E48009 Bilbao, Basque Country – Spain}
\email{nemethi.andras@renyi.mta.hu }

\title{On the  topology of elliptic singularities}

\begin{document}

\keywords{normal surface singularity,
resolution  graph, rational homology sphere,
elliptic singularities, elliptic sequence, Seiberg--Witten invariant, surgery formula,
Poincar\'e series, geometric genus, periodic constant}

\subjclass[2010]{Primary. 32S05, 32S25, 32S50, 57M27
Secondary. 14Bxx, 14J80}

\begin{abstract}
For any elliptic normal surface singularity with rational homology sphere link we
 consider a  new elliptic sequence, which differs   from the one introduced by
 Laufer and S. S.-T. Yau.
 However, we show that their length coincide. Using the properties of both sequences  we succeed
 to connect the common length with the geometric genus and also with several topological invariants,
 e.g. with the Seiberg--Witten invariant of the link.
\end{abstract}

\maketitle

\linespread{1.2}


\begin{center}
{\em  Dedicated to Gert--Martin Greuel}
\end{center}

\pagestyle{myheadings} \markboth{{\normalsize  J. Nagy, A. N\'emethi}} {{\normalsize On the topology of elliptic singularities }}


\section{Introduction}\label{s:intr}
\subsection{} The most important analytic invariant of a complex normal surface singularity $(X,o)$ is
its geometric genus $p_g$. Even if we fix a topological type --- usually identified by the link $M$
of the germ, or by a resolution graph ---, and even if we assume that the link $M$ is a rational homology sphere,
the geometric genus might vary when we vary the analytic structure. Hence, it is natural to find
 topological bounds for it. In the literature there are  several topological invariants, which are related with $p_g$
 in this sense.

 One of them is $\Pam$, cf. \cite{trieste,NS16,NO17}, see subsection \ref{s:Pathi} below.
 It is a topological upper bound for $p_g$, that is, for any analytic structure one
 has $p_g\leq \Pam$. However, usually it is hard to verify whether the inequality is
 optimal or not for a certain topological type, that is, whether a special analytic structure realizes the equality. (One knows topological types when the inequality is not sharp, see Example \ref{ex:pGpathno}.)

 Another topological invariant is the (modified) Seiberg--Witten invariant
$ \overline{\ssw}_{0}(M))$
 of the link  (associated with the canonical $spin^c$--structure). It is related with the geometric genus via the Seiberg--Witten Invariant Conjecture (SWIC) $p_g(X,o)=\overline{\ssw}_{0}(M))$, cf.
\cite{trieste,NCL,NO08,NOk,NS16,NWCasson}, which is expected to be true  for certain special analytic structures. But, again, the verification of this identity usually is hard (and in some cases it is not even true).

In the case of elliptic singularities there is another topological numerical invariant, the length of the elliptic
sequence introduced by Laufer and S. S.-T. Yau \cite{Yau5,Yau1}. In the numerically Gorenstein case
(when a Gorenstein structure exist) it is easier to connect with $p_g$ and  $\Pam$, however in the general case the
Yau's elliptic sequence is rather complicated 
 (and it is also hard to connect with possible analytic realizations).

In order to eliminate these difficulties, 
 we consider  a new elliptic sequence, which in the non--numerically Gorenstein case is different than the one studied by Yau, and which fits much better in such  comparisons.
It was motivated (and  introduced) in  the author's study of the Abel map of  surface singularities \cite{NNIII}, and it has several advantages compared with the earlier approaches.
E.g., it identifies the support of a numerically Gorenstein subgraph with the following property.
If the analytic type supported  on this subgraph is Gorenstein, that $p_g(X,o)$  is maximal, and it satisfies the identity
$p_g(X,o)=\Pam$ (and the statement (3) from below).

In this note first we prove that the length  $\ell+1$ of the Yau's elliptic sequence coincides with the length
$m+1$ of our elliptic sequence. Then using properties of both sequences we prove the following statements for any elliptic germ with rational homology sphere link:

\vspace{1mm}

(1) $m+1=\ell+1$;

\vspace{1mm}

(2) $m+1=\Pam$;

\vspace{1mm}

(3) there exists an analytic structure (characterized precisely) supported on the fixed
elliptic topological type such that $p_g=m+1$;

\vspace{1mm}

(4) $m+1=\overline{\ssw}_{0}(M))$, in particular, for any analytic structure from (2) the SWIC holds.

\vspace{2mm}

Strictly speaking, in the proof of (2) we use an additional assumption, namely that the 
minimal resolution is good. The main reason for this assumption is that 
the elliptic sequences are defined (and have nice properties) in the minimal resolution, 
while the invariant $\Pam$ is defined in via good resolutions. We expect that the 
statement remains valid in any case, but in this note we did not check the compatibility of the two resolutions (the minimal one and the minimal good one) from the point of view of these two set of invariants (and  we didn't carry out the pathological cases either).

\subsection{} The structure of the article is the following. In section \ref{s:prel} we review the standard
notations related with resolution of normal surface singularities, and  we recall some facts regarding $\Pam$.
In the next section we discuss elliptic singularities (we always assume that the link
is a rational homology sphere).
We recall the definition of the elliptic sequence according to Yau, we establish several properties which will be needed later.
Then we discuss the special case of numerically Gorenstein graphs, and finally we provide the definition and several
properties of the `new' elliptic sequence. Finally in Theorem \ref{th:main} we prove (1) and (2).

Section \ref{s:surg} reviews several results regarding surgery properties of the Seiberg--Witten invariants
(based on some coefficient counting of the topological Poincar\'e series),  and
in the last section we prove (3) via such a surgery formula.

\section{Preliminaries and notations}\label{s:prel}
\subsection{Notations regarding a resolution}  \cite{Nfive,trieste,NCL,LPhd,NN1}
Let $(X,o)$ be the germ of a complex analytic normal surface singularity.
We denote by  $p_g$ the geometric genus of $(X,o)$.
{\it We will assume that the link $M$ of $(X,o)$ is a rational homology sphere.}

Let $\phi:\widetilde{X}\to X$ be a   resolution   of $(X,o)$ with
 exceptional curve $E:=\phi^{-1}(0)$,  and  let $\cup_{v\in\calv}E_v$ be
the irreducible decomposition of $E$.

$L:=H_2(\widetilde{X},\mathbb{Z})$, endowed
with a negative definite intersection form  $(\,,\,)$, is a lattice. It is
freely generated by the classes of  $\{E_v\}_{v\in\mathcal{V}}$.
 The dual lattice is $L'={\rm Hom}_\Z(L,\Z)=\{
l'\in L\otimes \Q\,:\, (l',L)\in\Z\}$. It  is generated
by the (anti)dual classes $\{E^*_v\}_{v\in\mathcal{V}}$ defined
by $(E^{*}_{v},E_{w})=-\delta_{vw}$ (where $\delta_{vw}$ stays for the  Kronecker symbol).
$L'$ is also  identified with $H^2(\tX,\Z)$. 

All the $E_v$--coordinates of any $E^*_u$ are strict positive.
We define the Lipman cone as $\calS':=\{l'\in L'\,:\, (l', E_v)\leq 0 \ \mbox{for all $v$}\}$.
As a monoid it is generated over $\bZ_{\geq 0}$ by $\{E^*_v\}_v$.
Write also $\calS:=\calS'\cap L$.

$L$ embeds into $L'$ with
 $ L'/L\simeq H_1(M,\mathbb{Z})$, which is abridged by $H$.
 The class of $l'$ in $H$ is denoted by $[l']$.

There is a natural (partial) ordering of $L'$ and $L$: we write $l_1'\geq l_2'$ if
$l_1'-l_2'=\sum _v r_vE_v$ with all $r_v\geq 0$. We set $L_{\geq 0}=\{l\in L\,:\, l\geq 0\}$ and
$L_{>0}=L_{\geq 0}\setminus \{0\}$.

The support of a cycle $l=\sum n_vE_v$ is defined as  $|l|=\cup_{n_v\not=0}E_v$.

Since $H_1(M,\Q)=0$,
each $E_v$ is rational, and the dual graph of any good resolution is a tree.

\bekezdes\label{ss:mincyc} {\bf Minimal cycles in $L'_{\geq 0}$ and in $\calS'$.}
Consider the semi-open cube $\{\sum_v l'_v E_v\in L' \ | \ 0\leq l'_v<1\}$.
 It contains a unique representative $r_h$ for every $h\in H$ so that $[r_h]=h$.
Similarly,   for any $h\in H$ there is a
 unique minimal element of $\{l'\in L' \ | \ [l']=h\}\cap \mathcal{S}'$,
 which will be denoted by $s_h$ (cf. Lemma \ref{lem:cs2} below).
One has  $s_h\geq r_h$;  in general, $s_h \neq r_h$.

\bekezdes\label{sss:s} {\bf A `Laufer--type' computation sequence targeting $\mathcal{S}'$.}
Recall the following fact:
\begin{lemma}\ \cite{Laufer72}, \cite[Lemma 7.4]{NOSZ}\label{lem:cs2} \
Fix any $l'\in L'$.
\begin{enumerate}
 \item[(1)] There exists a unique minimal element $s(l')$ of $(l'+L_{\geq 0})\cap \mathcal{S}'$.
 \item[(2)] $s(l')$ can be found via the following
  computation sequence $\{z_i\}_i$ connecting $l'$ and $s(l')$:
   set $z_0:=l'$, and assume that $z_i$ ($i\geq 0$) is already constructed. If
   $(z_i,E_{v(i)})>0$ for some $v(i)\in\mathcal{V}$ then  set $z_{i+1}=z_i+E_{v(i)}$. Otherwise
   $z_i\in \mathcal{S}'$ and necessarily  $z_i=s(l')$.
\end{enumerate}
\end{lemma}
In general the choice of the individual vertex $v(i)$ might not be unique, nevertheless the final
output $s(l')$ is unique.

If we start with an arbitrarily chosen $l'=E_v$ then $s(l')$ is the minimal (fundamental)
cycle $Z_{min}$ of $L$, that is, the minimal element of $\cS\setminus \{0\}$
 \cite{Artin62,Artin66,Laufer72}.
 In this case, the
sequence from part (2) usually is called the `{\it Laufer's computation sequence of $Z_{min}$}'.

Similarly, for any $h\in H$, if $l'=r_h$ then $s(l')=s_h$.

\subsubsection{}
The {\it (anti)canonical cycle} $Z_K\in L'$ is defined by the
{\it adjunction formulae}
$(Z_K, E_v)=(E_v,E_v)+2$ for all $v\in\mathcal{V}$.
(It is the first Chern class of the dual of the line bundle $\Omega^2_{\tX}$.)
We write $\chi:L'\to \Q$ for the combinatorial expression $\chi(l'):= -(l', l'-Z_K)/2$.

The singularity (or, its topological type) is called numerically Gorenstein if $Z_K\in L$.
(Since $Z_K\in L$ if and only if the line bundle $\Omega^2_{X\setminus \{o\}}$ of holomorphic 2--forms
on $X\setminus \{o\}$ is topologically trivial, see e.g. \cite{Du}, the $Z_K\in L$ property
is independent of the resolution). $(X,o)$ is called Gorenstein if $Z_K\in L$ and
$\Omega^2_{\tX}$ (the sheaf of holomorphic 2--forms) is isomorphic to $ \calO_{\tX}(-Z_K)$ (or,
equivalently, if the line bundle $\Omega^2_{X\setminus \{o\}}$ is holomorphically trivial).

Recall that if $\tX$ is a minimal resolution then (by the adjunction formulae) $Z_K\in \calS'$. 
In particular, $Z_K-s_{[Z_K]}\in L_{\geq 0}$.

\begin{lemma}\label{lem:szk}\ \cite[Lemma 2.1.4]{NNIII} \ Consider the minimal resolution
$\tX$ of $(X,o)$.
Then  $p_g=0$  whenever $Z_K=s_{[Z_K]}$.
If $Z_K>s_{[Z_K]}$ then $p_g=h^1(\calO_{Z_K-s_{[Z_K]}})$.
More generally,
$h^1(\tX, \calL)=h^1(Z_K-s_{[Z_K]}, \calL)$  for any $\calL\in \pic(\tX)$ with $c_1(\calL)\in -\calS'$.
\end{lemma}
\begin{proof} By generalized Kodaira or Grauert--Riemenschneider  vanishing
$h^1(\widetilde{X}, \mathcal{O}_{\widetilde{X}}(-\lfloor Z_K\rfloor))=0$).
Hence, if $\lfloor Z_K\rfloor)=0$ then $p_g=0$. Otherwise, using the exact sequence
$0\to \mathcal{O}_{\widetilde{X}}(-\lfloor Z_K\rfloor)\to \mathcal{O}_{\widetilde{X}}
\to \mathcal{O}_{\lfloor Z_K\rfloor}\to 0$ we get $h^1(\calO_{\lfloor Z_K\rfloor})=p_g$.
Next, consider the computation sequence from Lemma \ref{lem:cs2}
applied for $l'=r_h$ and show  by induction that
$h^1(\calO_{Z_K-z_i})=p_g$.

More generally, $h^1(\tX,\calL)=h^1(Z_K-z_i,\calL)$ for any $i$ by similar argument.
\end{proof}

\subsection{The invariant $\Pam$}\label{s:Pathi} Assume that at this time $\tX$ is the 
minimal good resolution. In this case, $Z_K\geq 0$, see e.g. \cite{Wim,PPP,NBOOK}. 
In particular, $\lfloor Z_K\rfloor \in L_{\geq 0}$. 
Let ${\mathcal K}$ be the (topologically defined) set of cycles 
$\lfloor Z_K\rfloor +L_{\geq 0}$.
Note that by a generalized Grauert--Riemenschneider vanishing \cite{GrRie}
 $h^1(\cO_Z)=p_g$ for any $Z\in {\mathcal K}$.

An {\it increasing path}  is a sequence of integral cycles
$\gamma:=\{l_i\}_{i=0}^t$, $l_i\in L$ such that $l_0=0$, $l_t\in {\mathcal K}$,  and for any $i<t$ one has $l_{i+1}=l_i+ E_{v(i)}$
for some $v(i)\in \cV$.  Denote
by $\cP^\uparrow$ the set of increasing paths. Moreover,
for any $\gamma\in\cPa$ and  $i<t$ define
\begin{equation}\label{eq:Bi}
p_i=\max\{0,\, \chi(l_{i})-\chi(l_{i+1})\}=\max\{0,(E_{v(i)},l_i)-1\}, \end{equation}
and set $S(\gamma):=\sum_{ i< t} p_i$ for any $\gamma\in\cPa$.
Furthermore, set
$\Pam:=\min_{\gamma\in \cPa}S(\gamma)$ as well.

The definition is mostly motivated by comparison of the geometric genus with
path lattice cohomology \cite{lattice}, see also \cite{NS16,NO17}.

\bekezdes \label{bek:pgpath} {\bf Upper bounds for the geometric genus.}
If $\gamma\in \cPa$ with $l_t=Z$  then
 $p_g=h^1(\cO_Z)$. Furthermore,   from the
exact sequence $0\to \cO_{E_{v(i)}}(-l_i)\to \cO_{l_{i+1}}\to \cO_{l_i}\to 0$  we get
\begin{equation*}
h^1(\cO_{l_{i+1}})-h^1(\cO_{l_i})\leq h^1(\cO_{E_{v(i)}}(-l_i)) = p_i\ \ \ \ \ (0\leq i<t).
\end{equation*}
In particular, for any analytic structure with the fixed resolution graph $\Gamma$ one has
\begin{equation}\label{eq:pgpath}
p_g\leq \Pam.\end{equation}
Equality 
holds if {\it for some} $\gamma\in\cPa$
 the above cohomology  exact sequences split for all $i$.
 The above inequality
$p_g\leq \Pam$ looks slightly artificial, even naive; nevertheless, for rather important
analytic structures along a well--chosen increasing path all the cohomology exact sequences  split, and the equality  $p_g= \Pam$ holds.
The equality  $p_g=\Pam$ is realized by the following analytic families (with rational homology sphere link):
(a) weighted homogeneous singularities;
(b)  superisolated singularities;
(c) Newton non-deg hypersurfaces;
(d) rational singularities;
(e) Gorenstein elliptic singularities.
(For details and further references see \cite{NS16,NBOOK}).

 In Theorem \ref{th:main}
 we will show that the equality in  $p_g\leq \Pam $ can be realized in the case of any
 non--numerically Gorenstein elliptic topological type  as well.

\begin{example}\label{ex:pGpathno} \  \cite{NO17}
On the other hand, one can find topological types of singularities
(even  with integral homology sphere link) such that for {\it any} analytic structure
the strict inequality $p_g<\Pam$ holds:
For the next graph
 $\Pam=4$, nevertheless for all analytic structures $2\leq p_g\leq 3$.

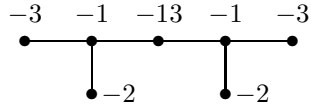
\begin{picture}(300,45)(30,0)
\put(125,25){\circle*{4}}
\put(150,25){\circle*{4}}
\put(175,25){\circle*{4}}
\put(200,25){\circle*{4}}
\put(225,25){\circle*{4}}
\put(150,5){\circle*{4}}
\put(200,5){\circle*{4}}
\put(125,25){\line(1,0){100}}
\put(150,25){\line(0,-1){20}}
\put(200,25){\line(0,-1){20}}
\put(125,35){\makebox(0,0){$-3$}}
\put(150,35){\makebox(0,0){$-1$}}
\put(175,35){\makebox(0,0){$-13$}}
\put(200,35){\makebox(0,0){$-1$}}
\put(225,35){\makebox(0,0){$-3$}}
\put(160,5){\makebox(0,0){$-2$}}
\put(210,5){\makebox(0,0){$-2$}}
\end{picture}

\end{example}

\subsection{Paths with fixed end--cycles} \label{bek:ineqPath}
We fix an arbitrary $Z\in L_{>0}$.
 We extend the above definition by taking paths $\gamma$ with
end--cycle $l_t$ exactly $Z$. Accordingly,
for any such fixed $Z$, we set $\Pam(Z):=
\min_\gamma S(\gamma)$, where $\gamma$ runs over all increasing pathes  with $l_0=0$ and $l_t=Z$.
By similar argument as in subsection \ref{bek:pgpath} one obtains for any $Z>0$
\begin{equation}\label{eq:pgpathZ}
h^1(\cO_Z)\leq \Pam(Z).\end{equation}

\begin{lemma}\label{lem:MINgamma}  The following facts hold:

(a) If $Z_1\leq Z_2$ then $\Pam(Z_1)\leq \Pam(Z_2)$ (this is true for 
any good resolution graph).

(b) Assume additionally that $\lfloor Z_K\rfloor>0$ (e.g., when the resolution graph
is minimal good). 
If $\lfloor Z_K\rfloor \leq Z$ then $\Pam(\lfloor Z_K\rfloor )=\Pam(Z)$.
In particular, $\Pam(\lfloor Z_K\rfloor )=\max_{Z>0} \Pam(Z)=\rm \Pam$.
\end{lemma}
\begin{proof}
(a) Take  $Z_2=Z_1+E_{v}$.
Fix a path   $\gamma=\{l_i\}_i$ with $l_t=Z_2$.
 Let $k$ be the largest index with $v(k)=v$.  For any
 $i\in\{k+1,\ldots, t\}$ write $l_i=\bar{l}_i+E_v$. Then $\bar{l}_{k+1}=l_k$ and $\bar{l}_t=Z_1$.
We will replace the path $\gamma$ by the path $\bar{\gamma}$ consisting of
$l_0,\ldots, l_k, \bar{l}_{k+2},\ldots, \bar{l}_t$. Note that
for $k+1\leq i< t$ one has
$$\chi(l_i)-\chi(l_{i+1})=\chi(\bar{l}_i)-\chi(\bar{l}_{i+1})+(E_v,E_{v(i)})\geq \chi(\bar{l}_i)-\chi(\bar{l}_{i+1}).$$
Hence $S(\bar{\gamma})\leq S(\gamma)$. Then use induction.

(b) For any $Z>\lfloor Z_K\rfloor $ there exists $E_v\subset |Z-\lfloor Z_K\rfloor |$ such that $\chi(Z-E_v)\leq \chi(Z)$. Indeed, if not, then $(Z-Z_K,E_v)\geq 0$ ($\dag$) for any
 $E_v\subset |Z-\lfloor Z_K\rfloor |$. Let $\{Z_K\}_1 $ be the part of $\{Z_K\}$ supported on
 $|Z-\lfloor Z_K\rfloor |$, and $\{Z_K\}_2=\{Z_K\}-\{Z_K\}_1$. Hence, from ($\dag$),
 $(Z-\lfloor Z_K\rfloor -\{Z_K\}_1,E_v)\geq (\{Z_K\}_2,E_v)\geq 0$. This can happen only if
 $Z-\lfloor Z_K\rfloor -\{Z_K\}_1=0$. But $Z$ is integral, hence $Z=\lfloor Z_K\rfloor$, a contradiction.
Therefore,  the path
  $\gamma$ which realizes $\Pam(\lfloor Z_K\rfloor)$ can be completed to a longer path from
  0 to $Z$ with the same $S(\gamma)$
  (construct inductively a decreasing path from $Z$ to $\lfloor Z_K\rfloor$ via the previous statement).
  Hence $\Pam(Z)\leq \Pam(\lfloor Z_K\rfloor)$.
  Then use (a).
\end{proof}

\section{Elliptic singularities. The elliptic sequences.}\label{s:elliptic}

\subsection{Elliptic singularities} \label{ss:ell}
 Let $Z_{min}\in L$ be the minimal  cycle.
 Recall that $(X,o)$ is called elliptic if
$\chi(Z_{min})=0$, or equivalently, $\min _{l\in L_{>0}}\chi(l)=0$ \cite{Laufer77,Wa70}.
It is known that if we decrease the decorations (Euler numbers), or we take a full subgraph
 of an elliptic graph, then we get either elliptic or a rational graph.

Let $C$ be the minimally elliptic cycle \cite{Laufer77,weakly}, that is, $\chi(C)=0$ and
$\chi(l)>0$ for any $0<l<C$. There is a unique cycle with this property, and if
$\chi(D)=0$ ($D\in L$) then necessarily $C\leq D$. In particular, $C\leq Z_{min}$.
In the sequel we assume that the resolution is minimal.  Then $Z_K\in \calS'$, hence in the numerically
Gorenstein case $Z_{min}\leq Z_K$ by the minimality of $Z_{min}$ in $\calS\setminus 0$.

The minimally elliptic singularities were introduced by Laufer in \cite{Laufer77}.  In a minimal
resolution they are characterized (topologically) by $Z_{min}=Z_K=C$.
Moreover, $(X,o)$ is minimally elliptic if and only if $p_g(X,o)=1$ and $(X,o)$
is Gorenstein.  For details see  \cite{Laufer77,weakly,Nfive}.

For an arbitrary elliptic singularity the minimally elliptic cycle $C$ supports a minimally elliptic singularity
(resolution graph). One has the following lemma of Laufer.

\begin{lemma}\label{lem:12} \ \cite{Laufer77}
Consider the minimal resolution of a minimally elliptic singularity.

(a) \  Let $\{z_i\}_{i=1}^t$ be a computation sequence of $Z_{min}$ with $z_1=E_v$ for some $v$.
Then $\chi(z_i)=1$ for all $i<t$, $(z_i, E_{v(i)})=1$ for all $i<t-1$,
and in the last step $(z_{t-1}, E_{v(t-1)})=2$.

(b) \ Fix any pair $E_0$ and $E_1$ ($E_0\not=E_1$) of irreducible exceptional divisors.
Then there exists a computation sequence for $Z_{min}$ which starts with
$E_1$ (i.e. $z_1=E_1$) and ends with $E_0$ (i.e. $E_{v(t-1)}=E_0$).
Moreover, let $E_0$ be an irreducible component whose coefficient  in $Z_{min}$ is greater than one.
Then there exists a computation sequence for $Z_{min}$ which starts  and ends with $E_0$.
\end{lemma}

\subsubsection{} {\bf Elliptic sequences.}
One of the most important tools in the study of elliptic singularities are the {\it elliptic sequences}.
It is  defined from the combinatorics of the resolution graph.
It  can be regarded also as a sequence of cycles with decreasing supports, or also as resolution graphs of
a sequence
of singularities obtained by contracting the exceptional divisors supported in the
corresponding cycles.
They were introduced by Laufer and S. S.-T. Yau, for the definition in the general
(non--Gorenstein) case see \cite{Yau5,Yau1}. In the numerically Gorenstein case the construction is
simpler, see additionally \cite{weakly,Nfive,OkumaEll} as well.

First  we recall the construction of the sequence in the general
(not necessarily numerically Gorenstein)  case according to S. S.-T.  Yau, and we list several properties
what we will need. Later  we will provide another elliptic sequence in the non--numerically
Gorenstein case, which was introduced in \cite{NNIII}, whose definition `adapts' the numerically Gorenstein case.
The length of both sequences serve as upper bounds for the geometric genus of any analytic structure supported
on the topological type identified by the graph.
The sequence from \cite{NNIII} differs from  the one introduced by Yau,
however, our goal is to prove that their length is the same.

\subsection{The elliptic sequence, the general case, according to S. S.-T. Yau}\label{ss:ElSeqGC}
For any non--zero reduced effective cycle   $D\in L_{>0}$
we write $Z_D$ for the minimal cycle
of the full subgraph determined by $|D|=D$.

\begin{definition}\ \label{def:EllSeqGen} \cite{Yau5}, \cite[Def. 3.3]{Yau1}
Let $E$ be the exceptional set of the minimal resolution $\phi:\tX\to X$ of an elliptic singularity.
Let $C$ be the minimally elliptic cycle.

If $(C,Z_{min})<0$ then the elliptic sequence consists of one element, namely $\{Z_{min}\}$.

If $(C,Z_{min})=0$, let $D_1$ be the maximal connected subvariety (reduced effective cycle)
of $E$ containing the support $|C|$,
such that $(E_v,Z_{min})=0$ for all $E_v\subset D_1$. Since $Z_{min}^2<0$, $D_1\not=E$.

Assume that the term $D_{i-1}$ of the elliptic sequence is already defined. If $(C,Z_{D_{i-1}})=0$,
let $D_i$ be the maximal connected subvariety of $D_{i-1}$, containing $|C|$, such that
$(E_v,Z_{D_{i-1}})=0$ for all $E_v\subset D_i$. Again, $D_i$ is properly contained in $D_{i-1}$.
This process stops after finitely many steps, say with $D_\ell$, which has the property
$(C,Z_{D_\ell})<0$.

Write $D_0:=E$, $Z_{D_0}:=Z_{min}$. Then the elliptic sequence is $\{Z_{D_0},\ldots, Z_{D_\ell}\}$.
Its length is $\ell+1$ ($\ell\geq 0$).
\end{definition}

\subsubsection{}
The sequence satisfies several properties, see e.g. \cite{Yau5,Yau1}. E.g., the next ones are immediate.
From the construction $|C|\subseteq D_\ell\varsubsetneq \cdots \varsubsetneq D_0$. Note also that
$C=Z_{|C|}$ (valid for minimally elliptic singularities). Hence   $C\leq Z_{D_\ell}< \cdots< Z_{D_0}$ too.

Moreover, by a general property of the minimal cycles,
 $h^0(\cO_{Z_{D_j}})=1$. On the other hand
  $h^1(\cO_{Z_{D_j}})\geq h^1(\cO_C)=1$ and  $\chi(Z_{D_j})\geq 0$ (ellipticity),
  hence  $\chi(Z_{D_j})=0$   for all $0\leq j\leq \ell$. Furthermore, from the construction,
   $(Z_{D_k},Z_{D_j})=0$ for any $k\not=j$.

\begin{lemma}\label{lem:ellseq}
For any $0\leq k \leq \ell$ set $F_{k}:=\sum_{i=0}^{k} Z_{D_i}$.
Then $\chi(F_k)=0$ and $F_k\in\cS$.
\end{lemma}

\begin{proof}
%
%

$\chi(F_k)=0$ follows from the above discussions. Next we prove $F_k\in\calS$.

If $E_v\subset |D_j|$ for all $j\leq k$ then $(E_v,Z_{D_j})\leq 0$ for all $j$, hence
$(F_k,E_v)\leq 0$.

Assume that $E_v\subset D_{j-1}$ but $E_v\not\subset D_j$ for some $1\leq j\leq k$.
 Then $(E_v,Z_{D_i})\leq 0$ for $i\leq j-1$ since $D_{j-1}\subset D_i$.
 If $(E_v,Z_{D_j})=0$, that is, $E_v$ does not intersect the support $D_j$, then $E_v$ does not
 intersect the smaller supports $\{D_i\}_{k\geq i\geq j}$ either, hence $(E_v , Z_{D_i})=0$ for all $i\geq j$.
Hence we are done again.

Next,  assume that $(E_v,Z_{D_j})>0$, hence $E_v$ intersects $D_j$, say along the component $E_u$.
Then we observe two facts. First, $(\dagger)$ $(E_v,Z_{D_{j-1}})<0$ since otherwise $E_v$ would be in $D_j$.
Second, $Z_{D_j}$ can be completed by a computation sequence to $Z_{D_0}$ by adding $E_v$
at the first step,
hence the multiplicity of $E_u$ in $Z_{D_j}$ should be 1 (by Laufer's algorithm, and from the fact that both $\chi(Z_{D_j})$ and $\chi(Z_{D_0})$ are zero).
Therefore,
 $(\ddagger)$ $(Z_{D_j}, E_v)= 1$. Then  $(\dagger)$ and $(\ddagger)$ imply
$(Z_{D_j}+Z_{D_{j-1}},E_v)\leq 0$. If $j=k$ then again we are done.

If $j<k$ then $(Z_{D_j},C)=0$ and $Z_{D_{j+1}}$ exists, and it is a summand of $F_k$.
We show that $(Z_{D_{j+1}},E_v)=0$.
This means that $E_v$ does not intersect the support $D_{j+1}$ hence neither the smaller supports
$\{D_i\}_{k\geq i\geq j+1}$, hence $(Z_{D_{i}},E_v)=0$ for all $i\geq j+1$.

Assume the opposite, that is,   $(Z_{D_{j+1}},E_v)>0$.
 Then necessarily  $(Z_{D_{j}},E_v)>0$ too.
Then consider the cycle $l:= Z_{D_{j+1}}+Z_{D_{j}}+E_v$. Then $\chi(l)=
\chi( Z_{D_{j+1}}+Z_{D_{j}})+1-( Z_{D_{j+1}}+Z_{D_{j}}, E_v)=
1-( Z_{D_{j+1}}+Z_{D_{j}}, E_v)<0$, a fact which contradicts the ellipticity of the graph.
\end{proof}

\begin{proposition}\label{prop:MINES}   \ Assume that the minimal resolution is good.
Then  $\Pam\leq  \ell+1 $.
\end{proposition}
\begin{proof}
Since $(C,Z_{D_m})<0$ there exists $E_0\subset |C|$ with $(E_0,Z_{D_m})<0$. By  Lemma \ref{lem:12}
there exists a computation sequence $\{z_i\}_{i=1}^t$ of $C=Z_{|C|}$ such that
$(z_i,E_{v(i)})=1$ for $i<t-1$ and $(z_{t-1},E_{v(t-1)})=2$, where $E_{v(t-1)}$ is exactly  $E_0$.
We mark this step by  ($\dagger$).
(The first cycle $z_1$ can be any base--cycle  $E_1$ from  $|C|$.)

This computation sequence can be completed to a computation sequence of $Z_{D_j}$, $\{z^{(j)}_i\}_{i=1}^{t(j)}$,
such that for $t\leq i< t(j)$ one has $(z^{(j)}_i,E_{v(i)})=1$ (for $\chi(Z_{D_j})=\chi(C)=0$).

If we concatenate these sequences, $\{z^{(0)}_i\}_i,\ \{z^{(1)}_i\}_i,\ \ldots, \{z^{(\ell)}_i\}$,
we get $\{z^{c}_i\}_{i=1}^{\sum_jt(j)}$, which connects 0 (or $E_1$) to
$F_\ell:=\sum_{j=0}^\ell Z_{D_j}$, and in it
exactly $\ell+1$ times happens that $\chi(z^c_{i+1})<\chi(z^c_i)$. When this happens then
$\chi(z^c_{i+1})=\chi(z^c_i)-1$, and they occur exactly when we add
the last component of $C$,
namely during steps marked by ($\dagger$).

Next, we continue the sequence $\{z^c\}_i$ with $F_\ell+\{z^c\}_i$.
Note that $F_\ell$ has two key properties: $(F_\ell,E_0)=(Z_{D_\ell},E_0)<0$ and
$F_\ell \in\calS$, cf. Lemma  \ref{lem:ellseq}. Therefore,
$(F_\ell+z^c_i,E_{v(i)})\leq (z_i^c,E_{v(i)})\leq 2$, and
$(F_\ell+z^c_i,E_{v(i)})$ might be 2 only at steps marked by ($\dagger$).
But at these steps
$(F_\ell,E_{v(i)})=(F_\ell, E_0)=(Z_{D_m}, E_0)<0$,
hence $(F_\ell+z^c_i,E_{v(i)})<2$ always, and
 the $\chi$-values along the sequence  $F_\ell+\{z^c\}_i$ are non-decreasing.
 This remains true
for $nF_\ell+\{z^c\}_i$ for any $n\geq 1$, hence we get an infinite sequence
$\{\ell_i\}_i$ whose multiplicities tend to infinity,
and which satisfies $\sum_i\max\{0,\chi(\ell_i)-\chi(\ell_{i+1})\}=\ell+1$. This proves the  inequality
$\Pam\leq  \ell+1$.
\end{proof}
\begin{corollary}\label{cor:Yau} For any analytic structure supported by an elliptic graph with length $\ell+1$ one has $p_g\leq \ell+1$.
\end{corollary}
\begin{proof} Combine (\ref{eq:pgpath}) with Proposition \ref{prop:MINES}. \end{proof}

\begin{remark}
S.S.-T. Yau in \cite{Yau5} considered another sequence, the `Laufer sequence',
and he proved that
$p_g$ is not greater than the length of the Laufer sequence.
On the other hand, J. Stevens in
(the first preprint version of) \cite{Stevens84} proved that  the elliptic sequence and the Laufer sequence coincide.
Hence, these two results imply the inequality of  Corollary \ref{cor:Yau}.
\end{remark}

\subsection{The elliptic sequence in the numerically Gorenstein case}
\label{ss:NGES} See also \cite{weakly,Nfive,OkumaEll}.

The elliptic sequence consists of a sequence of integral cycles
$\{Z_{B_j}\}_{j=0}^m$, where $Z_{B_j}$ is the minimal
 cycle supported on  the connected  reduced cycle $B_j$.  $\{B_j\}_{j=0}^m$ are defined inductively as follows.  For $j=0$ one takes $B_0=E$, hence $Z_{B_0}=Z_{min}$. Then $C\leq Z_{min}=Z_{B_0}\leq Z_K$.
 If $Z_{B_0}=Z_K$ then we stop, $m=0$, this situation corresponds to the minimally elliptic case.

 Otherwise one takes $B_1:=|Z_K-Z_{B_0}|$. One  verifies that   $|C|\subseteq B_1\varsubsetneq B_0$,
 $B_1$ is connected,
 and it supports a numerically Gorenstein elliptic
 topological type with canonical cycle $Z_K-Z_{B_0}$.
 (Furthermore, $(E_v,Z_{B_0})=0$ for any $E_v\subset B_1$. The proof of all these facts are
  similar to the proof of Lemma \ref{lem:b0} below.)
  In particular, $C\leq Z_{B_1}\leq
 Z_K-Z_{B_0}$. Then we repeat the inductive argument. If $Z_{B_1}=  Z_K-Z_{B_0}$, then we stop,
 $m=1$. Otherwise,  we define $B_2:=|Z_K-Z_{B_0}-Z_{B_1}|$. $B_2$ again is connected, $|C|\subseteq B_2
 \varsubsetneq B_1$, and supports a  numerically Gorenstein elliptic
 topological type with canonical cycle $Z_K-Z_{B_0}-Z_{B_1}$. After finite steps we get
 $Z_{B_m}=Z_K-Z_{B_0}-\cdots-Z_{B_{m-1}}$, hence the minimal cycle and the canonical cycle on
 $B_m$ coincide. This means that $B_m$ supports a minimally elliptic singularity with $Z_{B_m}=C$.

 We say that the length of the elliptic sequence $\{Z_{B_j}\}_{j=0}^m$ is $m+1$.

It is also convenient to introduce the notations
$$C_j=\sum_{i=0}^j Z_{B_i} \ \ \mbox{and} \ \  C'_j=\sum_{i=j}^m Z_{B_i}\  \ \  (0\leq j\leq m).$$
By these notations, $C_0=Z_{min}$, $C'_m=C$, and $C_m=C'_0=Z_K$.

The next lemma summarizes the immediate properties of the elliptic sequence.

\begin{lemma}\label{e211}

(a)\ $B_0=E,\ B_1=|Z_K-\zbz|,\ B_2=|Z_K-\zbz-\zbe|,\ \ldots, \ B_m=|C|; $
each $B_j$ is connected and the inclusions $B_{j+1}\subset B_j$ are strict.
Moreover, $Z_{min}=\zbz\supset\zbe\supset \cdots\supset \zm=C$.

(b)\ If $E_v\subset B_{j+1}$ then $(E_v,\zj)=0$ for all $v$ and $j$.
In particular, $(\zi,\zj)=(C_i,\zj)=0$ for all $0\leq i<j\leq m$.

(c)\ $Z_K=\sum_{i=0}^m\zi$.

(d)\ $(E_v,C'_i)=(E_v,Z_K)$ for any $E_v\subset |C'_i|$. In other words,
$C'_i$ is the canonical cycle of $|C'_i|=B_i$. 

(e)\ $C_i\in \cS$. 
\end{lemma}

\begin{proof} {\it (a)--(d)} follow from the construction. The proof of
{\it (e)} is as follows.
If $E_v\subset B_i$ then $(E_v,\zj)\leq 0$ for any $j\leq i$,
 hence $(E_v,C_i)\leq 0$.
If $E_v\not\subset B_i$ then $(E_v,C_i)=(E_v,Z_K-C'_{i+1})$. Now, $(E_v,Z_K)\leq
0$ (by the minimality of the resolution) and $(E_v,C'_{i+1})\geq 0$
(because $|C'_{i+1}|\subset B_i$). \end{proof}

\begin{proposition}\label{prop:EllSeqthesame} Fix a numerically Gorenstein elliptic
minimal graph.
Consider  the elliptic sequence $\{D_j\}_{j=0}^{\ell}$
defined in \ref{def:EllSeqGen} and $\{B_j\}_{j=0}^{m}$ defined in \ref{ss:NGES}.
 Then $m=\ell$ and $D_j=B_j$ for any $j$. In particular (cf. Corollary \ref{cor:Yau}),
 $p_g\leq m+1$.
\end{proposition}
\begin{proof}
Clearly, $D_0=B_0=E$. Moreover, the continuation of both sequences is decided by the same
criterion:  by \ref{e211}{\it (b)} one has $(Z_{D_0},C)=0\ \Leftrightarrow\ Z_K>Z_{min}$.
  Next we show that
 $D_1=B_1$. From \ref{e211}{\it (b)} we get $B_1\subset D_1$. Assume that $B_1\not=D_1$.
 Since $D_1$ is connected,
 then there exists $E_v\subset D_1$, in the support of $D_1-B_1$,  such that
 $(E_v,B_1)>0$. Then, $(E_v,Z_{B_0})=0$, but $(E_v,Z_K-Z_{B_0})=
 (E_v, \sum_{i\geq 1}Z_{B_i})>0$.
 Hence $(E_v,Z_K)>0$, a fact which contradicts with the minimality of the resolution.
Then we proceed by induction.
\end{proof}
\begin{remark}\label{rem:Gorexists}
Any numerically Gorenstein topological type admits a Gorenstein analytic structure
\cite{PPP}. Hence, any numerically Gorenstein elliptic topological type
is realized by a  Gorenstein elliptic analytic structure. For analytic characterizations
of such structures see \cite{weakly}. One of the characterizations is
that $(X,o)$ is Gorenstein if and only if $p_g=m+1$.
Hence, the Gorenstein structure are exactly those ones which realizes the
 maximal $m+1$.
\end{remark}

\subsection {\bf The elliptic sequence in the  non--numerically Gorenstein case, according to \cite{NNIII}.} \label{ss:NNIII}
Assume that $Z_K\not\in L$, that is, $r_{[Z_K]}\not=0$. Since the resolution is minimal,
$Z_K\in \calS'$, hence $Z_K\geq s_{[Z_K]}$. Since the graph is not rational,
by Lemma \ref{lem:szk} $Z_K> s_{[Z_K]}$.
We will use the following notations: $B_{-1}:=E$, $Z_{B_{-1}}:=s_{[Z_K]}$ and $B_0:=|Z_K-s_{[Z_K]}|$.
(Note that $Z_{B_{-1}}\in L'\setminus L$.)
\begin{lemma}\label{lem:b0}\ \cite{NNIII}
 (a) $B_0$ is connected, $C\subseteq B_0\varsubsetneq E$,  and
$(E_v,Z_{B_{-1}})=0$ for any $E_v\subset B_0$.

(b) $B_0$ supports a numerically Gorenstein
elliptic topological type with canonical cycle $Z_K-s_{[Z_K]}$.
\end{lemma}
For the convenience of the reader we insert the proof from \cite{NNIII} here as well.
\begin{proof}
{\it (a)}
Write $l:=Z_K-s_{[Z_K]}$. Then $\chi(s_{[Z_K]})=\chi( Z_K-l)=\chi(l)$.  Since $(X,o)$ is elliptic
$\chi(s_{[Z_K]})=\chi(l)\geq 0$ $(\dag)$. Also, $(s_{[Z_K]},l)\leq 0$
since $s_{[Z_K]}\in\calS'$ $(\ddag)$.
On the other hand, $0=\chi(Z_K)=\chi(l+s_{[Z_K]})=\chi(l)+\chi(s_{[Z_K]})-(l,s_{[Z_K]})$.
Then by $(\dag)$ and $(\ddag)$ the expressions from the right hand side are $\geq 0$, hence
necessarily $\chi(l)=(l,s_{[Z_K]})=0$. If $l$ has more connected components, say
$\cup_il_i$, then $\chi(l_i)=0$ for all $i$, hence each $l_i$  contains/dominates  a minimally
elliptic cycle (cf. \cite{Laufer77}), a fact which contradicts  the uniqueness of the minimally
elliptic cycle. Hence $|l|=B_0$ is connected and $|C|\subset B_0$. Furthermore,
$(l, s_{[Z_K]})=0$ shows that $|l|\not=E$.

{\it (b)} $C\subseteq B_0\varsubsetneq E$ shows that $\min_{|l|\subset B_0, \, l>0}
 \chi(l)=0$, hence $B_0$ supports an elliptic topological type. Moreover, from   $(l, s_{[Z_K]})=0$
 we read that for any $E_v$ from the support of $l$ one has
$(E_v,s_{[Z_K]})=0$, hence $(E_v,Z_K-s_{[Z_K]})=(E_v,Z_K)$, hence $Z_K-s_{[Z_K]}\in L$ is
the canonical cycle on $B_0$.
\end{proof}
Then, as a continuation of the sequence, starting from $B_0$ and its integral canonical class $Z_K-s_{[Z_K]}$  we construct
the sequence $\{Z_{B_j}\}_{j=0}^m$ as in the numerically Gorenstein case.

We say that the elliptic sequence $\{Z_{B_j}\}_{j=-1}^m$ has length $m+1$ and `pre--term'
$Z_{B_{-1}}=s_{[Z_k]}\in L'$.

In order to have a uniform notation, in the numerically Gorenstein case we set
$Z_{B_{-1}}:=0$ (which, in fact, it is $s_{[Z_k]}$).
In any case, from  above (see also \cite[2.11]{weakly}), for latter references,
\begin{equation}\label{eq:orthogonal}
(E_v,Z_{B_j}) \ \ \mbox{for any $E_v\subset B_{j+1}$} \ \ \ (-1\leq j< m).
\end{equation}
Set
$C_t:=\sum_{i=-1}^t Z_{B_i}$ and $C_t':=\sum_{i=t}^m Z_{B_i}$, $-1\leq t\leq m$.
E.g. $C_m=Z_K$ and, in general,  $C'_j$ is the canonical cycle  of $B_j$.
Furthermore, $\chi(Z_{B_j})=\chi(C_j)=\chi(C_j')=0$.

\begin{example}\label{ex:new}
Consider the next elliptic graph

\begin{picture}(200,50)(-120,0)
\put(70,30){\circle*{4}}\put(90,30){\circle*{4}}\put(110,30){\circle*{4}}\put(130,30){\circle*{4}}
\put(50,30){\circle*{4}}\put(130,30){\circle*{4}}\put(150,30){\circle*{4}}
\put(-30,30){\circle*{4}}
\put(110,40){\makebox(0,0){\small{$-3$}}}\put(130,40){\makebox(0,0){\small{$-3$}}}
\put(150,20){\makebox(0,0){\small{$E_1$}}}
\put(130,20){\makebox(0,0){\small{$E_2$}}}
\put(-10,30){\circle*{4}}
\put(10,30){\circle*{4}}
\put(30,30){\circle*{4}}
\put(10,10){\circle*{4}}
\put(-30,30){\line(1,0){180}}\put(10,10){\line(0,1){20}}
\end{picture}

\noindent where the $(-2)$--vertices are  unmarked.
 $Z_K$ and $s_{[Z_K]}$ are

 \begin{picture}(500,45)(70,0)

\put(70,30){\makebox(0,0){\tiny{14/3}}}
\put(90,30){\makebox(0,0){\tiny{28/3}}}
\put(110,30){\makebox(0,0){\tiny{42/3}}}
\put(110,10){\makebox(0,0){\tiny{21/3}}}
\put(130,30){\makebox(0,0){\tiny{35/3}}}
\put(150,30){\makebox(0,0){\tiny{28/3}}}
\put(170,30){\makebox(0,0){\tiny{21/3}}}
\put(190,30){\makebox(0,0){\tiny{14/3}}}
\put(210,30){\makebox(0,0){\tiny{7/3}}}
\put(230,30){\makebox(0,0){\tiny{4/3}}}
\put(250,30){\makebox(0,0){\tiny{2/3}}}

\put(300,30){\makebox(0,0){\tiny{2/3}}}
\put(320,30){\makebox(0,0){\tiny{4/3}}}
\put(340,30){\makebox(0,0){\tiny{6/3}}}
\put(340,10){\makebox(0,0){\tiny{3/3}}}
\put(360,30){\makebox(0,0){\tiny{5/3}}}
\put(380,30){\makebox(0,0){\tiny{4/3}}}
\put(400,30){\makebox(0,0){\tiny{3/3}}}
\put(420,30){\makebox(0,0){\tiny{2/3}}}
\put(440,30){\makebox(0,0){\tiny{1/3}}}
\put(460,30){\makebox(0,0){\tiny{1/3}}}
\put(480,30){\makebox(0,0){\tiny{2/3}}}
\end{picture}

\noindent $B_0$ is obtained by deleting $E_1$ from $E$, while
$B_1$ by deleting $E_1$ and $E_2$, hence $B_1=|C|$. The length is  $m+1=2$.
Furthermore,
$C_{-1}=s_{[Z_K]}$, $C_0=s_{[Z_K]}+ Z_{B_0}$ and
$C_1=Z_K=s_{[Z_K]}+
Z_{B_0}+Z_{B_1}$; they are not integral cycles.

On the other hand, $D_0=E$ and $D_1=|C|$ (since $(Z_{min}, E_2)<0$).
$F_0=Z_{min}$ (which equals $Z_{B_0}+E_1$)
 and $F_1=Z_{min}+C$. These are integral cycles. The length is $\ell+1=2$.
\end{example}

In the above example $E_2$ from the support of
$B_0$ satisfies $(Z_{min},E_2)<0$. This is a general phenomenon,
a fact, which provides the `starting bridge' between the two elliptic
sequences $\{D_j\}_j$ and $\{B_j\}_j$.

\begin{proposition}\label{prop:Bo}
(a)  There exists $E_v$ in the support of $B_0$ with  $(E_v, Z_{min})<0$.

(b) Any numerically Gorenstein connected subgraph is contained in $B_0$.
In particular, the largest numerically Gorenstein connected subgraph is $B_0$.
\end{proposition}
\begin{proof}
(a) Though the statement is topological, it is convenient to fix a special analytic structure on $(X,o)$, which produces a very fast and elegant proof. Since $Z_{min}\in\calS$,
there exists an analytic structure for which this cycle is realized
as a divisor of $f\circ \phi$ for a certain  function $f$ \cite{P01}.
Assume that
$(E_v, Z_{min})=0$ for any $E_v\subset B_0$. Then the strict transforms of $\{f=0\}$
do not intersect $B_0$, hence $\calO_{\tX}(-Z_{min})|_{Z_K(B_0)}$ is trivialized by $f$.
Therefore, using Lemma \ref{lem:szk} for the structure  sheaf, we get that
$h^1(Z_k-s_{[Z_K]}, \calO_{\tX}(-Z_{min}))=h^1(Z_k-s_{[Z_K]},\calO_{Z_k-s_{[Z_K]}})=p_g$.
On the other hand, using Lemma \ref{lem:szk} for $\calO_{\tX}(-Z_{min})$ we get
$h^1(Z_k-s_{[Z_K]}, \calO_{\tX}(-Z_{min}))=h^1(\tX,\calO_{\tX}(-Z_{min})) $.
In particular, $h^1(\tX,\calO_{\tX}(-Z_{min})) =p_g$.

Now consider the long cohomology exact sequence associated with $0\to \calO_{\tX}(-Z_{min})\to \calO_{\tX}
\to \calO_{Z_{min}}\to 0$, and using the well--know fact that
$H^0(\calO_{\tX})\to H^0( \calO_{Z_{min}})=\C$ is onto, we get that
$h^1(\calO_{\tX}(-Z_{min}))=p_g-h^1(\calO_{Z_{min}})=p_g-1$.
This leads to a contradiction.

(b)
Let $I$ be a connected  support of a numerically Gorenstein subgraph.
Let the canonical cycle on $I$ be $Z \in L$. Then   $(Z_K- Z , E_v) = 0$ for all $E_v \subset |Z|$.
Else,  if  $E_v \not\subset  |Z|$,  we have $(Z_K, E_v) \leq 0$ and $(Z, E_v) \geq 0$,
so $(Z_K - Z, E_v) \leq 0$.
This means, that $Z_K - Z \in S'$.
Therefore,
 $Z_K-Z\geq s_{[Z_K]}$. This reads as $Z\leq Z_K-s_{[Z_K]}$, or $I\subset B_0$.
\end{proof}

\begin{remark}\label{rem:Gorenstein}
For a non--numerically Gorestein graph one can consider both elliptic sequences,
namely $\{D_j\}_{j=0}^\ell$ and $\{B_j\}_{j=0}^m$.
For the first one we know from Corollary \ref{cor:Yau} that
 $p_g\leq \ell+1$. For the second one we know from
 Lemma \ref{lem:szk} that  $p_g=h^1(\calO_{B_0})$ and also $h^1(\calO_{B_0})\leq m+1$,
 cf. \cite{weakly}. Furthermore, we know that on $Z_K(B_0)$ the maximal
 $p_g=m+1$ can be realized by (any) Gorenstein structure, see also Remark
 \ref{rem:Gorexists}.
 (This is one of the main advantages of the sequence $\{B_j\}_{j=0}^m$:
 it produces numerically Gorenstein supports, and the elliptic length of the
 numerically Gorenstein support $B_0$ coincides with the length of $\Gamma$.)
 The Gorenstein analytic structure of $B_0$ (or of  a small
 tubular neighbourhood of $\cup_{v\in B_0}E_v$ in $\tX$) can be extended to an
 analytic structure of $\tX$. This shows that the graph $\Gamma$  supports an
 analytic structure with $p_g=m+1$.

 The analogous statement for $\{D_j\}_{j=0}^\ell$ (which guarantees the existence
 of any analytic structure with $p_g=\ell+1$) is not clear yet. This will be
 a consequence of the next theorem.
\end{remark}

\begin{theorem}\label{th:main}
 Fix any (not necessarily numerically Gorenstein) elliptic graph
(associated with a minimal resolution). Then

(a) $m+1 = \ell+1$.

(b) In particular, there exists an analytic type for which $p_g = \ell+1$.

(c)  Assume that the minimal resolution is good. Then 
$m+1 = \ell+1= \Pam$. Therefore, the general
topological upper bound $p_g\leq \Pam$ for $p_g$
in the case of  elliptic singularities is sharp: the equality can be realized
by some analytic structure.
\end{theorem}
\begin{proof}
The direction $m+1 \leq  \ell+1$ follows from the discussion from
Remark \ref{rem:Gorenstein}: there exists an analytic structure with
 $p_g=m+1$; hence from Lemma \ref{cor:Yau} one has  $m+1=p_g\leq \ell+1$.

We prove $m+1 \geq  \ell+1$ by induction on $\ell$. If $\ell= 0$ it is trivial.
Next assume that $\ell > 0$ and we know the statement for singularities with
`$D$--length' $\ell$.

By Proposition \ref{prop:Bo} $B_0\not\subset D_1$ ($\dag$). Next, denote by
$B_0(D_1)$ the `$B_0$ term' of the elliptic sequence associated with $D_1$.
Then  $B_0(D_1)\subset B_0$. Indeed, $B_0(D_1)$  is included in $D_1$ and it is a
connected  numerically Gorenstein support, hence by Proposition \ref{prop:Bo}{\it (b)},
 $B_0(D_1)\subset B_0$. But this inclusion should be strict. Indeed,
 $B_0\not=  B_0(D_1) $ contradicts ($\dag$).
 Hence $B_0(D_1)\varsubsetneq B_0$.

 Let us denote by $maxp_g(D)$ the maximum $p_g$ which can be realized by different analytic structures supported on a connected support/subgraph $D$.

Now, $maxp_g(D_1)=maxp_g(B_0(D_1))$ by Lemma \ref{lem:szk}.
Since $B_0(D_1)\subset B_0$ we have $maxp_g(B_0(D_1))\leq maxp_g(B_0)$. However, since
$B_0$ is a numerically Gorenstein support, and its maximal $p_g$ is realized by a
Gorenstein structure, which has the property that its cohomological cycle is
exactly its canonical cycle with support $B_0$, any smaller support has strict smaller
$maxp_g$. Since $B_0(D_1)\varsubsetneq B_0$, we get that
$maxp_g(B_0(D_1))< maxp_g(B_0)=m+1$.

On the other hand, the $D$-length of $D_1$ is $\ell$ (since the $D$--elliptic
sequence of $D_1$ is $\{D_1, \ldots, D_\ell\}$). Hence for $D_1$ the inductive step works.
In particular, $maxp_g(D_1)=\ell$. This combined with the statements from the previous
paragraph gives $m+1=maxp_g(B_0)>maxp_g(B_0(D_1))=maxp_g(D_1)=\ell$. That is, $m+1\geq \ell+1$.

From  $m+1= \ell+1$ and Remark \ref{rem:Gorenstein}
 we get that there exists an analytic structure with
  $p_g=\ell+1$. This combined with
  $p_g\leq \Pam \leq \ell+1$ (valid for any analytic structure, cf. (\ref{eq:pgpath}) and
Proposition   \ref{prop:MINES}) we get $maxp_g(E)=\Pam=\ell+1$.
\end{proof}

\begin{remark}\label{rem:UNIVERSAL} Both elliptic sequences
 $\{D_j\}_{j}$ and $\{B_j\}_j$ have some geometric universal properties.
For more information (and proofs) the reader is
invited to consult the references below. Here we mention only the
next chosen ones (they will be not applied  in this form in this paper,  though
some related partial statements were already used).

\vspace{1mm}

(a) \cite{NBOOK} \ If $l\in\calS$ and $\chi(l)=0$ then $l\in\{0, F_0, \ldots, F_\ell\}$.

(b) \cite{NNIII} \ If $l'\in \calS'$, $[l']=[Z_K]$, and $l'\leq Z_K$ then
$l'\in\{C_{-1}, C_0,\ldots C_m\}$.

(c) \cite{NNIII} \ The support of any numerically Gorenstein connected subgraph
belongs to $\{B_i\}_{i=0}^m$.
\end{remark}

\section{Review of surgery formulae for the  Seiberg--Witten invariant}\label{s:surg}

We fix a complex normal surface singularity $(X,o)$ and one of its good resolutions $\phi:\tX\to X$.
In the sequel we will review some topological invariants associated with the link $M$ and with the resolution graph
$\Gamma$ (or, with the lattice $L$). We will adopt all the notations of Section \ref{s:prel}.
In particular, we will assume that $M$ is a rational homology sphere. We will write also $M=M(\Gamma)$,
where we think about it as the plumbed manifold associated with $\Gamma$.
For more information and more details see \cite{CDGPs,CDGEq,NJEMS,NN1,BN,LNNSurg,LNNDual}.
For an overview see also \cite{ICM,NBOOK}.

\subsection{The Seiberg--Witten invariants of the link}\label{ss:SW}
The smooth oriented 4--manifold $\tX$ admits several $spin^c$--structures.
 Let $\widetilde{\sigma}_{can}$ be the {\it canonical
$spin^c$--structure on $\widetilde{X}$} identified by $c_1(\widetilde{\sigma}_{can})=-K$.
Furthermore, let  $\sigma_{can}\in \mathrm{Spin}^c(M)$
 be its restriction to $M$, called the {\it canonical
$spin^c$--structure on $M$}.  $\mathrm{Spin}^c(M)$ is an $H$--torsor, hence
the number of $spin^c$--structures supported on the oriented 3--manifold $M$ is $|H|$.
In this note we will focus only on the canonical one.

 We denote by $\mathfrak{sw}_{\sigma}(M)\in \bQ$ the
\emph{Seiberg--Witten invariant} of $M$ indexed by the $spin^c$--structures $\sigma\in {\rm Spin}^c(M)$ (cf. \cite{Lim,Nic04}).
(We will use the sign convention of \cite{BN,NJEMS}.) Again, in this note we focus merely on the SW--invariant
associated with the canonical $spin^c$--structure, $\ssw_{can}(M) $.

In fact, it is more convenient (imposed by surgery formulae)
 to use the {\it modified Seiberg--Witten invariant} defined by
\begin{equation}\label{eq:MSW}
\overline{\ssw}_{0}(M)= -\frac{K^2 + |\calv|}{8} - \ssw_{\sigma_{can}}(M).\end{equation}
There are  several combinatorial expressions established for the Seiberg--Witten invariants.  
 For rational homology spheres,
Nicolaescu \cite{Nic04} showed  that $\mathfrak{sw}(M)$ is
equal to the Reidemeister--Turaev torsion normalized by the Casson--Walker invariant. In the case when $M$ is a negative
definite plumbed rational homology sphere, combinatorial formula for Casson--Walker invariant in terms of the plumbing graph can be found in Lescop
\cite{Lescop}, and  the Reidemeister--Turaev torsion is determined by N\'emethi and Nicolaescu \cite{NN1} using Dedekind--Fourier sums.

A different  combinatorial formula of $\{\mathfrak{sw}_\sigma(M)\}_\sigma$ was proved  in \cite{NJEMS} using qualitative properties of the coefficients of the topological
multivariable series (`zeta function') $Z(\mathbf{t})$. This note also will exploit this connection further.

\subsection{The topological Poincar\'e series $Z(\bt)$}\label{ss:Zt}
The  \emph{multivariable topological Poincar\'e series} is the
Taylor expansion $Z(\mathbf{t})=\sum_{l'} z(l')\bt^{l'}
\in\bZ[[L']] $ at the  origin of the `rational function'
\begin{equation}\label{eq:1.1}
f(\mathbf{t})=\prod_{v\in \mathcal{V}} (1-\mathbf{t}^{E^*_v})^{\delta_v-2},
\end{equation}
where
$\bt^{l'}:=\prod_{v\in \mathcal{V}}t_v^{l'_v}$  for any $l'=\sum _{v\in \mathcal{V}}l'_vE_v\in L'$ ($l'_v\in\bQ$).
It has a natural and unique decomposition according to the elements of
$h\in H$ defined by
 $Z(\mathbf{t})=\sum_{h\in H}Z_h(\mathbf{t})$, where $Z_h(\mathbf{t})=\sum_{[l']=h}z (l')\bt^{l'}$.
 Corresponding to the choice of the canonical $spin^c$--structures here we make the choice of the series
 $Z_0(\bt)$ associated with $h=0$. In this  subseries $Z_0(\bt)$ of $Z(\bt)$  all the exponents belong to $L$
 (hence, it is a `genuine' series).    The expression
(\ref{eq:1.1}) shows that  $Z(\mathbf{t})$ is supported in the \emph{Lipman cone} $\mathcal{S}'$,
in particular $Z_0(\bt)$ is supported in $\calS=\calS'\cap L$.

Recall that  all the entries of
$E_v^*$ are strict positive, hence for any
$x\in L$, $\{l'\in \calS'\,:\, l'\not\geq x\}$ is finite. In particular the next  `counting function'
of the coefficients of $Z_h$ ($h\in H$) is well--defined:
\begin{equation}\label{eq:countintro}
Q_h:\{ x\in L'\,:\, [x]=h\} \to \bZ, \ \ \ \
Q_h(x)=\sum_{l'\ngeq x,\, [l']=h} z(l').
\end{equation}
The point is that for $x$ `sufficiently deeply inside of the Lipman cone' the function $x\mapsto Q_h(x)$
behaves as a quasipolynomial $\mathfrak{Q}_h(x)$. Furthermore,  the  values $\mathfrak{Q}_h(0)$
(indexed by all $h\in H$) provide the modified Seiberg--Witten invariants of the link (indexed by the $spin^c$--structures)  \cite{NJEMS}.
E.g.,  $\mathfrak{Q}_0(0)$ is  exactly $\overline{\ssw}_{0}(M)$.
The value $\mathfrak{Q}_0(0)$ is called the `periodic constant'
of the series $Z_0(\bt)$. In this note we try to bypass the theory of periodic constants and the theory of
quasipolynomials associated with counting functions, since in the final arguments we will not need them;
in this overview we mention them just to show the line of ideas behind the scenes.

The point is that important surgery formulae are also formulated in terms of  `periodic constants'
\cite{BN,LNNSurg,LNNDual}. Here we will recall the most general (and recent) one.

\subsection{A surgery formula}\label{ss:SURG}\cite{LNNSurg}
 Let $\cali\subset \calv$ be an arbitrary non--empty subset of $\calv$,  and write
 $\calv\setminus \cali$ as the union of full
  connected subgraphs  $\cup_i\Gamma_i$. Then one has the following  formula:
  \begin{equation}\label{eq:surg}
  \overline{\ssw}_{0}(M(\Gamma))-\sum_i \ \overline{\ssw}_{0}(M(\Gamma_i))= {\rm pc} (Z_0(\bt_{\cali})),
  \end{equation}
where $Z_0(\bt_{\cali})$ is the series with reduced variables defined as $Z_0(\bt_{\cali}):=
Z_0(\bt)|_{t_v=1, v\not\in \cali}$, and $ {\rm pc} (Z_0(\bt_{\cali}))$ is its periodic constant.

Since the periodic constant is determined by a complicated regularization procedure using the
asymptotic behaviour of the counting function of the coefficients of the corresponding series,
usually it  is hardly computable. This is the reason 
why is desired to find a replacement for it.
The next formula determines it in terms of a concrete  finite sum (precise evaluation of the `dual' counting
function). Behind this result the key ingredients are the $H$--equivariant  multivariable Ehrhart  theory
 of quasipolynomials associated with the above Poincar\'e series \cite{LN,LPhd}, and the
Ehrhart--Macdonald--Stanley equivariant reciprocities
(combined with  the duality
  of $L'$ and the series $Z(\bt)$ induced by $l\leftrightarrow Z_K-l$).
The next identity, proved in \cite[Theorem 4.4.1 (b)]{LNNDual}, shows that
${\rm pc} (Z_0(\bt_{\cali}))$ equals the value of the counting function 
associated with the 
coefficients of $Z_{[Z_K]}(\bt_{\cali})$ evaluated at $Z_K$:
\begin{equation}\label{eq:dual}
{\rm pc} (Z_0(\bt_{\cali}))=Q_{[Z_K],\cali}(Z_K):=
 \sum_{l'|_\cali \ngeq Z_K|_{\cali},\ [l'] = [Z_K] } z(l'),
\end{equation}
where $l'|_{\cali}$ is the projection of $l'$ to the variables $v\in \cali$ (if $l'=\sum_vr_vE_v$ then
$l'|_{\cali}=\sum_{v\in\cali}r_vE_v$).

\section{The Seiberg--Witten invariant of links of elliptic singularities}\label{s:SW}

\subsection{The canonical SW invariant of the plumbed manifold of an elliptic graph} \label{ss:swel}\

We fix an elliptic graph $\Gamma$ as in Section \ref{s:elliptic} and we will use all the notations of that section.

Above we discussed already two topological invariants of $M$ (or $\Gamma$), namely the length of the elliptic sequence
(defined in two different ways), $m+1=\ell+1$, and also $\Pam$. Theorem \ref{th:main} established their
coincidence. The previous section introduced a third invariant, namely   $\overline{\ssw}_{0}(M(\Gamma))$.

\begin{theorem}\label{th:SWmain}  $\overline{\ssw}_{0}(M(\Gamma))=m+1$.
%
\end{theorem}
\begin{proof}
In the proof we will use an inductice procedure based on the structure of
the elliptic sequence $\{B_j\}_{j=-1}^m$ from subsection \ref{ss:NGES} and \ref{ss:NNIII}. (We also write $B_{m+1}:=\emptyset$.)

The proof is given in two steps separating the numerically and non--numerically Gorenstein cases.

{\bf Case 1.} Assume that $\Gamma$ is numerically Gorenstein. We will use induction on
$m\geq 0$. If $m=0$ then the graph is minimally elliptic. In this case any analytic structure supported on $\Gamma$ is Gorenstein with $p_g=1$,  and they are also splice quotients. Hence for them the Seiberg--Witten Invariant
Conjecture from \cite{NN1} holds, that is,  $\overline{\ssw}_{0}(M(\Gamma))=p_g$ (proved in \cite{NO08,NCL}),
hence $\overline{\ssw}_{0}(M(\Gamma))=1$. Otherwise, the  $m=0$
 case can also be proved by adopting the next inductive argument by comparing the
  graph supported on $B_0$ by the empty set.

Next, we run induction. Assume that the statement is already proved for a
graph with length $m$ and we
fix some elliptic $\Gamma$ with length $m+1$. We fix $I:=B_0\setminus B_1$. Since $M(\Gamma(B_1))$ is minimally elliptic with length $m$, we know from the inductive step that    $\overline{\ssw}_{0}(M(\Gamma(B_1)))=m$.
On the other hand, from (\ref{eq:surg}) and (\ref{eq:dual}) we have
$$\overline{\ssw}_{0}(M(\Gamma(B_0)))=\overline{\ssw}_{0}(M(\Gamma(B_1)))+
Q_{[Z_K],\cali}(Z_K).$$
Hence, we need to show that $Q_{[Z_K],\cali}(Z_K)=1$.
In this numerically Gorenstein case $Z_K\in L$, hence $[Z_K]=0\in H$, and the expression of
 $Q_{[Z_K],\cali}(Z_K)$
from the right hand side of
(\ref{eq:dual})  becomes
$\sum z(l)$,  summed over $l\in L$ with $l|_\cali \ngeq Z_K|_{\cali}$.

Since $Z_0$ is supported in $\calS$, any $l\not=0$ in the support of $Z_0$ has the property that $l\geq Z_{min}$.
On the other hand, along $\cali=B_0\setminus B_1$ we have $Z_K|_{\cali}=Z_{min}|_{\cali}$, cf.
Lemma \ref{e211}{\it (c)}. This means that any $l\not=0$ form the support of $Z_0$ satisfies $l|_{\cali}\geq Z_K|_{\cali}$. In particular, in the sum only one term is non--zero, namely the one corresponding to $l=0$ with $z(0)=1$.

{\bf Case 2.} Assume that $\Gamma$ is an elliptic graph of length $m+1$ with  $Z_K\not\in L$.
Now we set $\cali:=B_{-1}\setminus B_0=E\setminus B_0$. Since $B_0$ supports a numerically Gorentein graph, from Step 1 we already know that $\overline{\ssw}_{0}(M(\Gamma(B_0)))=m+1$. We wish to show that
$\overline{\ssw}_{0}(M(\Gamma(B_{-1})))=m+1$ too. Hence from the surgery formula (\ref{eq:surg}) we need to verify  that
the following sum is zero:
$$ \sum_{l'|_\cali \ngeq Z_K|_{\cali},\ [l'] = [Z_K] } z(l').$$

Now, we know that $z(l') = 0$ unless $ l' \in S'$.  However,  if $ l' \in S'$ and $ [l'] = [Z_K]$, then $l' \geq s_{[Z_K]}$.  But $(s_{[Z_K]})|_\cali = Z_K|_{\cali}$ (cf. \ref{ss:NNIII}). This reads as
 $l'|_{\cali} \geq Z_K|_{\cali}$ for any relevant $l'$,
which means that the above summation is summed over the empty set.
\end{proof}

\begin{remark}\label{rem:SWIC}
For an normal surface singularity $(X,o)$ with link $M$ we say that the Seiberg--Witten Invariant Conjecture (SWIC)
is satisfied if $p_g(X,o)=\overline{\ssw}_{0}(M))$. For details and several examples see
\cite{trieste,NCL,NO08,NOk,NS16,NWCasson}. By our Theorem \ref{th:main} and \ref{th:SWmain} we obtain that
SWIC is satisfied by any
elliptic singularity with rational homology sphere link,
such that the restriction of the analytic structure to $B_0$ is Gorenstein.
(The identity $\ell+1=\overline{\ssw}_{0}(M))$ can be proved using the
techniques of the lattice cohomology and graded roots as well, since the Seiberg--Witten
invariant can also be realized as  the Euler characteristic of the lattice cohomology,
 cf. \cite{NOSZ,trieste,lattice,NOSZ,NJEMS,NBOOK}. The identity $p_g=\ell+1$  was known in the Gorenstein elliptic case, cf. Remark \ref{rem:Gorexists}.)
The present proof shows the `power' of the combination of \cite{LNNSurg} and \cite{LNNDual}: they provide a surprisingly short proof for the $p_g$--formula in this elliptic case. 
\end{remark}


\begin{thebibliography}{30}



\bibitem[A62]{Artin62} Artin, M.:
Some numerical criteria for contractibility of curves on algebraic surfaces.
{\em  Amer. J. of Math.}, {\bf 84}, 485-496, 1962.

\bibitem[A66]{Artin66} Artin, M.:
On isolated rational singularities of surfaces.
{\em Amer. J. of Math.}, {\bf 88}, 129-136, 1966.



\bibitem[BN10]{BN} Braun, G. and N\'emethi, A.:
Surgery formula for Seiberg--Witten invariants of negative definite
plumbed 3--manifolds,
{ \em J. f\"ur die reine und ang. Math.} {\bf 638} (2010), 189--208.



\bibitem[CDGZ04]{CDGPs} Campillo, A.,  Delgado, F. and Gusein-Zade, S. M.:
Poincar\'e series of a rational surface singularity, {\em Invent. Math.} {\bf 155} (2004),
no. 1, 41--53.

\bibitem[CDGZ08]{CDGEq}  Campillo, A.,  Delgado, F. and Gusein-Zade, S. M.:
Universal abelian covers of rational
surface singularities and multi-index filtrations,
{\em Funk. Anal. i Prilozhen.} {\bf 42} (2008), no. 2, 3--10.






\bibitem[Du78]{Du}
Durfee, A.H.: The signature of smoothings of complex surface singularities,
  Math. Ann. {\bf 232}, no. 1 (1978), 85-98.





\bibitem[GR62]{GRa} Grauert, H.: \"Uber Modifikationen und exzeptionelle
analytische Mengen,   {\it Math. Ann.} {\bf 146} (1962), 331--368.



\bibitem[GrRie70]{GrRie} Grauert, H. and Riemenschneider, O.:
Verschwindungss\"atze f\"ur analytische
kohomologiegruppen auf komplexen R\"aumen, {\it Inventiones math.}
{\bf 11} (1970), 263--292.












\bibitem[L13]{LPhd} L\'aszl\'o, T.: Lattice cohomology and
Seiberg--Witten invariants of normal surface singularities, PhD. thesis,
Central European University, Budapest, 2013.

\bibitem[LN14]{LN} L\'aszl\'o, T. and N\'emethi, A.: Ehrhart theory
of polytopes and Seiberg-Witten invariants of plumbed 3--manifolds,
{\em Geometry and Topology} {\bf 18} (2014), no. 2, 717--778.


\bibitem[LNN17]{LNNSurg} L\'aszl\'o, T., Nagy, J.  and N\'emethi, A.:
Surgery formulae for the Seiberg--Witten invariant
of plumbed 3--manifolds, arXiv:1702.06692.

\bibitem[LNN18]{LNNDual} L\'aszl\'o, T., Nagy, J.  and N\'emethi, A.:
 Combinatorial duality for Poincaré series, polytopes and invariants of plumbed 3-manifolds,
   arXiv:1805.03457.





\bibitem[La72]{Laufer72} Laufer, H.B.: On rational singularities,
{\em Amer. J. of Math.}, {\bf 94}, 597-608, 1972.



\bibitem[La77]{Laufer77} Laufer, H.B.: On minimally elliptic singularities,
{\em Amer. J. of Math.} {\bf 99} (1977), 1257--1295.



\bibitem[Les96]{Lescop} Lescop, C.: Global surgery formula for
the Casson--Walker invariant, {\em Ann. of Math. Studies}  {\bf 140},
Princeton Univ. Press, 1996.










\bibitem[Lim00]{Lim} Lim, Y.: Seiberg--Witten invariants for 3--manifolds in the case $b_1=0$ or $1$,
{\em Pacific J. of Math.} {\bf 195} (2000), no. 1, 179--204.





\bibitem[NN18]{NNIII} Nagy, J., N\'emethi, A.:
The Abel map for surface singularities  II. Elliptic germs,

\bibitem[N99]{weakly} N\'emethi, A.: ``Weakly'' Elliptic Gorenstein
singularities of surfaces,
{\em Inventiones math.},  {\bf 137}, 145-167 (1999).


\bibitem[N99b]{Nfive} N\'emethi, A.: Five lectures on normal surface singularities,
lectures at the Summer School in {\em Low dimensional topology} Budapest,
Hungary, 1998; Bolyai Society Math. Studies {\bf 8} (1999), 269--351.

\bibitem[N05]{NOSZ} N\'emethi, A.:
On the Ozsv\'ath--Szab\'o invariant of negative
definite plumbed 3--manifolds,
{\em Geometry and Topology} {\bf 9} (2005), 991--1042.

\bibitem[N07]{trieste} N\'emethi, A.: Graded roots and singularities,
{\em Singularities in geometry and topology},  World
Sci. Publ., Hackensack, NJ (2007), 394--463.

\bibitem[N08]{NPS} N\'emethi, A.: Poincar\'e series associated with
surface singularities, in Singularities I, 271--297,
{\em Contemp. Math.} {\bf 474}, Amer. Math. Soc., Providence RI, 2008.

\bibitem[N08b]{lattice} N\'emethi, A.:  Lattice cohomology of normal surface singularities
 {\em Publ. RIMS. Kyoto Univ.}, {\bf 44} (2008),  507--543.

\bibitem[N12]{NCL} N\'emethi, A.: The cohomology of line bundles
of splice--quotient singularities,
{\em Advances in Math.} {\bf 229} 4 (2012), 2503--2524.

\bibitem[N11]{NJEMS} N\'emethi, A.: The Seiberg--Witten invariants
of negative definite plumbed 3--manifolds,
{\em J. Eur. Math. Soc.} {\bf 13} (2011), 959--974.

\bibitem[NBook]{NBOOK}  N\'emethi, A.: Normal surface singularities,
book in preparation.

\bibitem[N18]{ICM}  N\'emethi, A.: Pairs of invariants of surface singularities,
Proceddings of ICM, Rio de Janeiro, 2018.

\bibitem[NN02]{NN1} N\'emethi, A. and Nicolaescu, L.I.:
 Seiberg--Witten invariants and surface singularities,
{\em Geometry and Topology} {\bf 6} (2002), 269--328.

\bibitem[NO09]{NOk} N\'emethi, A. and Okuma, T.:
On the Casson invariant conjecture of Neumann--Wahl,
{\em Journal of Algebraic Geometry} {\bf 18} (2009), 135--149.

\bibitem[NO08]{NO08} N\'emethi, A. and Okuma, T.: The Seiberg-Witen invariant conjecture for splice-quotients,
(joint paper with T. Okuma), {\em Journal LMS}
{\bf 28} (2008),  143-154.

\bibitem[NO17]{NO17} N\'emethi, A. and Okuma, T.:
Analytic singularities supported by a specific
integral homology sphere link, arXiv:1711.03384;
to appear in the
Proceedings dedicated to H. Laufer's 70th birthday (Conference at Sanya).

\bibitem[NS16]{NS16} N\'emethi, A., Sigurdsson, B.:
The geometric genus of hypersurface singularities,
  {\it  Journal of European Math. Soc.}   18 (2016), 825--851.








\bibitem[NW90]{NWCasson}
W. Neumann and J.~Wahl, Casson invariants of links of singularities,
Comment. Math. Helvetici {\bf 65} (1990), 58--78.


\bibitem[NW05]{NWsq}
Neumann, W. and Wahl, J.: Complete intersection singularities of splice type as universal abelian covers,
{\em Geom. Topol.} {\bf 9} (2005), 699--755.

\bibitem[NW10]{NWECTh}
W.~D. Neumann and J.~Wahl,
\emph{ The End Curve Theorem for normal complex surface singularities},  J. Eur. Math. Soc. {\bf 12} (2010), 471--503.


\bibitem[Nic04]{Nic04} Nicolaescu, L.:
Seiberg--Witten invariants of rational homology $3$--spheres,
{\em Comm. in Cont. Math.} {\bf 6} no. 6 (2004), 833--866.




\bibitem[O05]{OkumaEll} Okuma, T.: Numerical Gorenstein elliptic singularities,
Mathematische Zeitschrift 249 (2005), Issue 1, 31--62.


\bibitem[O08]{Ok} Okuma, T.: The geometric genus of splice--quotient singularities,
{\em Trans. Amer. Math. Soc.} {\bf 360} 12 (2008), 6643--6659.









\bibitem[P01]{P01} Pichon, A.: Fibrations sur le cercle et surfaces complexes,
{\it Ann. Inst. Fourier} {\bf 51} (2001), 337--374.


\bibitem[PP11]{PPP} Popescu-Pampu, P.: Numerically Gorenstein surface singularities are
homeomorphic to Gorenstain ones,  {\it Duke Math. Journal} {\bf 159} No. 3 (2011), 539-559.










\bibitem[St84]{Stevens84} Stevens, J.: Elliptic Surface Singulariites and Smoothings of Curves, Math. Ann.   {\bf 267} (1984), 239--249.



\bibitem[To85]{Tomari85} Tomari, M.: A $p_g$--formula and elliptic singularities,
{\it Publ. Res. Inst. Math. Sci. } {\bf 21} (1985), no. 2, 297--354.


\bibitem[V04]{Wim} Veys, W.: Stringy invariants of normal surfaces, 
{\it J. Algebraic Geom.}
{\bf 13} (2004), 115--141.

\bibitem[Wa70]{Wa70} Wagreich, Ph.: Elliptic singularities of surfaces, {\it Amer. J. of Math.},
{\bf 92} (1970), 419--454.


\bibitem[Y79]{Yau5} Yau, S. S.-T.: On strongly elliptic singularities,
{\em Amer. J. of Math.}, {\bf 101} (1979), 855-884.

\bibitem[Y80]{Yau1} Yau, S. S.-T.: On maximally elliptic singularities,
{\em Transactions of the AMS}, {\bf 257} Number 2 (1980), 269-329.

\end{thebibliography}
\end{document}